\documentclass[11pt]{amsart}
\usepackage{amsfonts,amssymb,latexsym,amsmath, amsxtra}
\usepackage[all]{xy}
\usepackage[dvips]{graphics}
\usepackage{color}
\theoremstyle{plain}
\newtheorem{theorem}{Theorem}[section]

\newtheorem*{conjecture*}{Conjecture}

\theoremstyle{definition}

\theoremstyle{remark}
\newtheorem*{remark}{Remark}

\numberwithin{equation}{section}

\newcommand{\N}{\mathbb N}
\newcommand{\Z}{\mathbb Z}
\newcommand{\C}{\mathbb C}

\def\({\left(}
\def\){\right)}

\newcommand{\abs}[1]{\left|#1\right|}

\def\k2{\frac{k}{2}}

\begin{document}

\title[Partial and Mock $q$-hypergeometric series]{
Partial theta functions and mock modular forms as $q$-hypergeometric series}

\dedicatory{Dedicated to the visionary Ramanujan, on the 125th anniversary of his birth.}

\author{Kathrin Bringmann}
\address{Mathematical Institute\\University of
Cologne\\ Weyertal 86-90 \\ 50931 Cologne \\Germany}
\email{kbringma@math.uni-koeln.de}

\author{Amanda Folsom}
\address{Department of Mathematics\\ Yale University\\ New Haven, CT 06520-8283 U.S.A.}
\email{amanda.folsom@yale.edu}

\author{Robert C. Rhoades}
\address{Department of Mathematics \\Stanford University \\ Stanford, CA 94305 \\ U.S.A. }
\email{rhoades@math.stanford.edu}


\thispagestyle{empty} ${\vspace{-2.5in}}$


\begin{abstract}
Ramanujan studied the analytic properties of many $q$-hypergeometric series.  Of those, mock theta functions have been particularly  intriguing, and by work of Zwegers, we now know how these curious $q$-series fit into the theory of automorphic forms.  The analytic theory of partial theta functions however, which have $q$-expansions resembling modular theta functions, is not well understood.  Here we consider families of $q$-hypergeometric series which converge in two disjoint domains. In one domain, we show that these series are often equal to one another, and define mock theta functions, including the classical mock theta functions of Ramanujan, as well as certain combinatorial generating functions,
as special cases.  In the other domain, we prove that these series are typically not equal to one another, but instead are related by partial theta functions.
\end{abstract}

\maketitle
\section{Introduction}\label{sec:curiousid}
  \subsection{Curious $q$-series}\label{curioussec}
   Partial theta functions have played a curious role in the theory of partitions (see for instance
\cite{alladi1, alladi2, bky}).
For example, Rogers proved the following identity for $|q|<1$:
\begin{equation}
\label{eqn:rogersExample} \sum_{n\geq 0}\frac{(-1)^n q^{\frac{n(n+1)}{2}}}{(-q)_n} =
\sum_{n\ge 0} \(\frac{-12}{n}\) q^{\frac{n^2 -1}{24}}=:\psi(q),
\end{equation}
where $(a)_n=(a; q)_n:=\prod_{j=0}^{n-1}\left(1-aq^j\right)$, and $\left(\frac{\cdot}{\cdot}\right)$ denotes the usual Kronecker symbol.  Andrews \cite{andrewsPartitions} provided a partition theoretic interpretation, giving a relationship between those partitions
into distinct parts and those with odd largest part and even largest part.
The series defined by $\psi(q)$ in (\ref{eqn:rogersExample}) is an example of a partial theta function, meaning that while the series expansion resembles that of a theta function, the sum runs only over half of a lattice instead of a full one.
Despite a closeness in the $q$-series expansions of partial theta functions and usual theta functions, which are modular forms,
the analytic theory of partial theta functions is not well understood.

Rogers's identity in (\ref{eqn:rogersExample}) is by no means an isolated example of this type.  There are many instances in the literature of $q$-hypergeometric series that are combinatorial generating functions, partial theta functions, mock theta functions, or modular forms.  For example, with $\abs{q}<1$, we have the combinatorial $q$-hypergeometric series \cite{atkinswinnerton}
\begin{align}\label{rankg}
R(w;q):= 1+ \sum_{n\geq 1}
\frac{q^{n^2}}{(wq)_n(w^{-1}q)_n}
= 1+\sum_{n \geq 1} \sum_{m \in \Z} R(m,n) w^mq^n,
\end{align}
where $R(m, n)$ denotes the number of integer partitions of $n$ with rank $m$,
and Dyson's \emph{rank} \cite{Dy44} of a partition equals the largest part minus its number of parts.
A well known theorem, using Durfee squares, gives that $R(1;q)=q^{1/24}/\eta(\tau)$ $(q=e^{2\pi i\tau}, \tau \in \mathbb H)$ is a modular form of weight $-1/2$.  Here $\eta(\tau):=q^{\frac{1}{24}}(q)_\infty$
is Dedekind's eta-function.
 Moreover, setting $w=-1$ in \eqref{rankg},
one obtains $f(q)$, Ramanujan's third order mock theta
function which he defined in his last letter to Hardy as
\begin{align}\label{eqn:fhypergeometric1}
f(q) := 1 + \sum_{n\geq 1} \frac{q^{n^2}}{(-q)_n^2},
\end{align}
which, in analogy to Rogers's partial theta function identity \eqref{eqn:rogersExample}, is a $q$-hypergeometric series admitting a combinatorial interpretation of $f$ as the generating function for the difference between the number of integer partitions with an even number of parts and those with an odd number of parts.
Thanks to the breakthrough of the 2002 doctoral thesis of Zwegers \cite{zwegers}, we now finally understand
how the mock theta functions fit into the theory of modular forms.
Loosely speaking, Zwegers showed that the mock theta functions satisfy modular transformation laws after they are ``completed" by the addition of certain non-holomorphic integrals.   Zwegers's discovery led to the development of a more overarching theory of {\emph{harmonic weak Maass forms}}, which are certain nonholomorphic modular forms.  Many advances have been made in the last decade in this area, notably by the first author and Ono, and their collaborators.   (See for instance, the surveys of Ono \cite{O}
and Zagier \cite{zagierBourbaki}
and the many references therein.)
On the other hand, while there are tantalizing hints which suggest a connection between certain partial theta functions and
Ramanujan's mock theta functions, the analytic theory of the partial theta functions is not well understood.  See, for instance, the works of Hikami, Lawrence-Zagier, and Zagier
\cite{hikami,lz,zagierStrange} on topological invariants of 3-manifolds.
In this paper, we seek to understand further curiosities evoked by partial theta functions and mock theta functions.
In particular, we study relationships between these functions inside \emph{and}
outside the unit disc $\mathbb D:= \{q \in \mathbb C : \ |q|<1\}$.
We first  illustrate this with a motivating example.

%%%%%%%%%%%%%%%
%%   EXAMPLE
%%%%%%%%%%%%%%%
\subsection{An example}\label{examplesec}  Consider the following $q$-hypergeometric series, defined for $|q|<1$ by
\begin{align}\label{fmock2} f^\star(q) &:= 1+\sum_{n\geq 1}\frac{(-1)^{n+1}q^n}{(-q)_n}.
\end{align}
It is known (see (26.22) of \cite{fine} for example)
that inside the unit disc $f^\star(q)$ equals Ramanujan's mock theta function $f(q)$, defined in \eqref{eqn:fhypergeometric1}.   On the other hand, $f^\star(q)$ also converges outside of the unit disc.  For
 $\abs{q}<1$ we have, using \eqref{randq} below (or equation (4.16) of \cite{hikami}), as well as (4.20) of \cite{hikami}
  \begin{align}\label{fstarid}f^\star\left(q^{-1}\right) &=
1\!-\!\!\!
\sum_{n\geq 1} (-1)^n \frac{q^{\frac{n(n-1)}{2}}}{(-q)_n}
 = 2\psi(q).
\end{align}

Note that the function $f^*(q)$ defined in \cite{hikami} is not equal to the function $f^\star(q)$ defined in (\ref{fmock2}). Continuing with this example, the $q$-hypergeometric series
$f(q)$ from \eqref{eqn:fhypergeometric1} also converges outside of the unit disc.  In light of \eqref{fstarid}, one might expect $f\left(q^{-1}\right)$ to equal $2\psi(q)$ (for $|q|<1$), but this is not the case. The difference between $f\left(q^{-1}\right)$ and the expected function $2\psi(q)$ is the series
\begin{equation}\label{spor}
S(q) := \frac{1}{(-q)_\infty^{2}}
 \sum_{n\ge 0} (-1)^n q^{\frac{n(n+1)}{2}},
\end{equation}
that is the quotient of a partial theta function and a modular form.  This fact is due to an identity found in Ramanujan's ``lost" notebook (given in \cite{andrewsLost}, valid for $|q|<1$):
$$f\left(q^{-1}\right)= 2\psi(q) - S(q).$$
We note that this identity may also be derived from \eqref{eqn:andrewsPartial} using \eqref{partialident}.

Consider now the series
\begin{align}
\label{phiq} \phi(q) &:= \sum_{n\ge 0} \frac{ q^{n^2}}{(-q^2;q^2)_n}.
\end{align}
Again this series converges inside and outside the unit disc.
In the region $\abs{q}<1$ it is one of Ramanujan's original (third order) mock theta functions.
 In his ``lost" notebook (see \cite{andrewsBerndt} p. 235), Ramanujan showed that for $|q|<1$
 $$
 \phi\left(-q^{-1}\right)=\psi(q).
 $$
 Hence, by \eqref{fstarid}, as before we might expect that the two third order mock theta functions
 satisfy $\phi(q) = \frac{1}{2}f(-q)$.
 However, this is again not the case.
 Interestingly, we know that the mock theta functions $\phi(q)$ and $f(-q)$ are related via
 the Mock Theta Conjectures (see for instance (3.12) of \cite{gmSurvey}). Namely
 $$2\phi(q) = f(-q) + T(q), $$
 where
 \begin{align}\label{mockphif} T(q) :=  \frac{(q^2;q^2)_\infty^7}{(q)_\infty^3 (q^4;q^4)_\infty^3}.\end{align}
The function $T(q)$  is, up to a power of $q$, a modular form.
We summarize the above discussion with the following diagram $(|q|<1)$:
\begin{equation}\label{diagramfq}
 \boxed{\begin{array}{llllllll} {f(q)} &&= &  f^{\star}(q) &  = & 2\phi(-q) - T(-q) & &  \text{\emph{mock theta}}
  \\ &&&&&&& \\
 \ \ \Big \downarrow &  & & \ \ \Big \downarrow &   & \ \ \Big \downarrow &&
  \\ &&&&&&& \\
f\left(q^{-1}\right)  + S(q) & & = & f^\star\left(q^{-1}\right)   & = & 2\phi\left(-q^{-1}\right) & & \text{\emph{partial theta}}
\end{array}}\end{equation}
That is, we have three different
$q$-hypergeometric series $f(q)$, $f^{\star}(q)$, and $\phi(q)$ (that are mock theta functions) which
define the same function inside $\mathbb D$, but different functions
outside $\mathbb D$.  The functions outside $\mathbb D$ include the same partial theta function $\psi(q)$,
and their difference  $S(q)$ is defined by a partial theta function and a modular form.
\begin{remark} Partial theta functions arise naturally in many $q$-series identities, and are often written in a slightly different form from that of $\psi(q)$ in (\ref{eqn:rogersExample}).  For example, it is not difficult to see that
$$\psi(q) = \sum_{n\geq 0} q^\frac{n(3n+1)}{2} \left(1-q^{2n+1}\right) = \sum_{n\geq 0} q^{\frac{n(3n+1)}{2}} - \sum_{n\leq -1} q^{\frac{n(3n+1)}{2}}.$$  From this perspective, such series are often called {\emph{false theta functions}} (see \cite{andrewsPartitions}).
\end{remark}

%%%%%%%%%%%%%
%%  FAMILIES
%%%%%%%%%%%%%
\subsection{$q$-hypergeometric families and mock theta functions}

In this paper, we study  families of $q$-hypergeometric series inside and outside $\mathbb D$,
that include the classical mock theta functions of Ramanujan, as well as certain combinatorial generating functions,
as special cases.  We show that in more generality, curiosities as depicted by (\ref{diagramfq}) also hold for the aforementioned families of functions.
 Throughout we denote
  \[
  (a_1, \ldots, a_r)_n=(a_1, \ldots, a_r; q)_n:=(a_1)_n\cdot\ldots\cdot(a_r)_n.
  \]

When $|q|<1$, we
consider the following families of $q$-hypergeometric series:
\begin{align}
g_2(w;q)&:=\sum_{n\geq 0} \frac{(-q)_n q^{\frac{n(n+1)}{2}}}{(w, w^{-1}q)_{n+1}},
 \label{eqn:universal2} \\
g_3(w;q)&:=\sum_{n\geq 0} \frac{q^{n(n+1)}}{(w, w^{-1}q)_{n+1}},
\label{eqn:universal3} \\
K(w;q) &:= \sum_{n\ge 0} \frac{(-1)^n q^{n^2}(q;q^2)_n}{(wq^2, w^{-1}q^2; q^2)_n},
\label{eqn:universalK}
\end{align}
which define functions (in certain ranges of $w$ which we will specify later) for both $\abs{q}<1$ and $\abs{q}>1$.
Hickerson, and later Gordon and McIntosh \cite{gmSurvey,hickerson1,hickerson2}, noticed that all of Ramanujan's original mock theta
functions can be written in terms of $g_2(w;q)$ and $g_3(w;q)$ for $\abs{q}<1$.
Hence, these functions are called \emph{universal mock theta functions}.
The function $K(w;q)$ appears in Ramanujan's ``lost" notebook \cite{andrewsLost}, and is related by modular transformation to the universal mock theta function (\ref{eqn:universal2}) \cite{mcintosh2}.  The families of $q$-hypergeometric series discussed here are essentially ``mock Jacobi forms", which first
occurred in Zwegers's thesis \cite{zwegers} and were systematically studied by the first author and Richter \cite{BR}.
For the purpose of this paper, we will not require this perspective.

The remainder of the paper is structured as follows.  In \S \ref{sec:analytic} we provide background and useful tools and
results from $q$-combinatorial analysis.
In \S \ref{sec:g3sec} - \S \ref{sec:Ksec} we
introduce and study many infinite families of $q$-hypergeometric series inside and outside $\mathbb D$ in analogy to the
curiosities observed in (\ref{diagramfq}).
%Our results exhibit a dual nature of these classes of $q$-hypergeometric series, and demonstrate how the non-uniqueness of $q$-hypergeometric representations for mock and false theta functions provide insight into their underlying analytic theories, while building a series of general examples to explain the
%structure of the additional terms like $S(q)$ that appear in the relations outside the unit disc.

%%%%%%%%%%%%%%%%%%%%
%%
%%  COMBINTORIAL ANALYSIS
%%
%%%%%%%%%%%%%%%%%%%%%
\section{$q$-combinatorial analysis}\label{sec:analytic}
In this section, we provide some combinatorial identities which we require for the proofs of the main theorems.
Many of the series studied in this paper are related to the \emph{basic hypergeometric series}, defined by
$$
F(a,b;t) = F(a,b;t,q) := \sum_{n\geq 0} \frac{(aq)_n}{(bq)_n}t^n.
$$
We refer the reader to  \cite{fine} for a detailed discussion on these functions, including questions of convergence.  These series satisfy various difference equations, including the following (see (2.4) of \cite{fine})
 \begin{align}\label{fineshift}F(a,b;t) = \frac{1-b}{1-t} + \frac{b-atq}{1-t} F(a,b;tq),
 \end{align}
 which lead to many interesting identities. One identity, the ``Rogers-Fine identity," that we will make use of here is as follows (see equation (1) of \cite{rogers}):
\begin{equation}\label{eqn:RogersFine}
F(a/q,b/q;t) = \sum_{n\geq0} \frac{(a, a t q/b)_n b^n
t^n q^{n^2-n} \left(1-a t q^{2n}\right)}{(b)_n (t)_{n+1}}.
\end{equation}  Iteration methods also give rise to the identities (see (12.2) and (12.3) of \cite{fine})
\begin{align}\label{122fine} (1-t) F(a,b;t) &= \sum_{n\geq 0} \frac{\left(\frac{b}{a}\right)_n}{(bq)_n (tq)_n}(-at)^n q^{\frac{n(n+1)}{2}}, \\ \label{123fine}
(1-t)F(0,b;t) &= \sum_{n\geq 0}\frac{(bt)^n q^{n^2}}{(bq)_n (tq)_n}, \end{align} where (\ref{123fine}) is derived from (\ref{122fine}) by letting $a\to 0$.
In special cases, these $q$-hypergeometric series may directly reduce to simpler expressions, such as (see (14.31) of \cite{fine})
\begin{align} (1-a) F(a,-a;a) = 1 + 2 \sum_{n\geq 1} (-1)^n a^{2n} q^{n^2}.\label{faaa} \end{align}
On the other hand, the following more involved result of Andrews (see Theorem 1 of \cite{andrewsPartial}),
when combined with the Rogers-Fine identity, is very useful in establishing many
partial theta function identities
\begin{align}\label{eqn:andrewsPartial}
\sum_{n \geq 0}
 \frac{(B, -Abq)_n }{(-aq, -bq)_n} q^n=&
-a^{-1}\frac{(B, -Abq)_\infty}{(-aq, -bq)_\infty}
\sum_{m \geq 0}\frac{\left(A^{-1}\right)_m }{\left(-Ba^{-1}\right)_{m+1} }\left(Aba^{-1}q\right)^m\\
& +\left(1+a^{-1}\right)(1+b)\sum_{m\geq 0}
\frac{\left(-a^{-1}q, -a^{-1} ABq\right)_m}{\left(-Ba^{-1}, Aba^{-1}q\right)_{m+1}}(-b)^m.\nonumber
\end{align}
Similarly, we have the following results of Ramanujan (see $(1.1)_R$ of \cite{andrewsPartial}, $(3.1)$ of \cite{by},  and $(3.11)_R$ of \cite{andrewsPartial}):
\begin{align} \label{ram11}
\sum_{n\geq 0} \frac{q^n}{(-aq, \!-a^{-1}q)_n}&=(1+a) \sum_{n\geq 0} \!a^{3n} q^\frac{n(3n+1)}{2} \!\left(1\!-\!a^2 q^{2n+1}\right)\!-\!\frac{a}{(-aq, \!-a^{-1}q)_\infty}\sum_{n\geq 0}(-1)^n a^{2n}q^{\frac{n(n+1)}{2}},
  \\
\label{by31}
\sum_{n\geq 0} \frac{(-1)^n a^{2n} q^{\frac{n(n+1)}{2}}}{(-aq)_n} &=  \sum_{n\geq 0} a^{3n} q^{\frac{n(3n+1)}{2}}\left(1-a^2 q^{2n+1}\right), \end{align}{\vspace{-.1in}} \begin{align}
\label{311Rand} (1\!+\!a^{-1})\!\!\sum_{n\geq 0} \!\frac{q^{2n + 1} \left(q;q^2\right)_n}{\left(-aq,-a^{-1}q;q^2\right)_{n+1}} = \sum_{n\geq 0} \!(-a)^n q^{\frac{n(n+1)}{2}} \!\!-\! (q)_\infty   \sum_{n\geq 0} \!q^{3n^2 + n} a^{3n} \left(1\!-\!a^2 q^{4n+2}\right) \!\!\left(\sum_{m \in \mathbb Z} q^{m^2}a^m\!\right)^{\!-1}\!\!\!\!\!.
\end{align}
%%%%%%%%%%%%%%%%%%%%%%%%%%%%%
%%%%%%%%%%%SECTION 3%%%%%%%%%%%%
%%%%%%%%%%%%%%%%%%%%%%%%%%%%%
\section{The universal mock theta function $g_3$}\label{sec:g3sec}
 In this section, we study the universal mock theta function $g_3(w;q)$ defined in (\ref{eqn:universal3}).
 Theorem \ref{thm:g3_q<1}
gives three different $q$-hypergeometric series representations of   $g_3(w;q)$ inside the unit disc.  Theorem \ref{thm:g3_q>1} shows that these series represent different functions outside $\mathbb{D}$ that are related to each other by partial theta functions.   To state our results, define for $|q|< 1$ and $w\in\C\setminus\{0\}$ with $w\neq q^\ell(\ell\in\mathbb Z)$
\begin{align*}
g_{3,1}(w;q) &:= g_3(w;q), \\
g_{3,2}(w;q) &:= -\frac{1}{w} + \frac{1}{w(1-w)} R(w;q).
\end{align*}
Additionally, for $w\neq q^{-\ell}(\ell\in\N)$ with $|w|>|q|$, we let
\[
g_{3,3}(w;q) := \sum_{n\ge 0} \frac{w^{-n} q^n}{(w)_{n+1}}.
\]
\begin{theorem}\label{thm:g3_q<1}
For $\abs{q}<1$ and $w \in \mathbb C$ with the same restrictions as above, we have that
$$
g_{3,1}(w;q) = g_{3,2}(w;q)= g_{3,3}(w;q).
$$
\end{theorem}
To describe our next result concerning these representations outside the unit disc (for certain ranges of $w$), we require further functions.  Generalizing the partial theta function
$\psi(q)$ defined in \eqref{eqn:rogersExample}, and the function $S(q)$, given in \eqref{spor}, we define for $\abs{q}<1 $ the (two-variable) functions
\begin{align}
\psi_1\left(w; q\right) &:= -w - w^2\sum_{n \geq 0} \(\frac{12}{n}\) w^{\frac{n-1}{2}}  q^{\frac{n^2-1}{24}} ,
\label{eqn:g3partial} \\
\psi_2(w;q) &:= \sum_{n\ge 0} w^{n} q^{\frac{n(n+1)}{2}}.  \nonumber
\end{align}
Moreover if $w\neq q^\ell (\ell\in\Z)$, let
\[
S_2\left(w;q\right) := \frac{w^2}{(wq, w^{-1})_\infty}\psi_2(-w^2;q).
\]
\begin{remark}Using the fact that $(-1)^{(n-1)/2} = \(\frac{-1}{n}\)$, we have that  $\psi_1(-1;q)=1-\psi(q)$.  Moreover $2S_2(-1;q) = S(q)$.
 \end{remark}
\begin{theorem}\label{thm:g3_q>1} For $\abs{q}<1$  and $w\in \C\setminus \{0\}$ with $w\neq q^\ell(\ell\in\Z)$   we have
\begin{equation}\label{eqn:g3_1/qA}
g_{3,2}\left(w;q^{-1}\right)  =\psi_1\left(w^{-1};q\right) + S_2\left(w^{-1};q\right).
\end{equation}
Moreover if $w\neq q^\ell(\ell\in\N_0)$ and $w\neq 0$, we have that
\begin{equation}\label{eqn:g3_1/qB}
 g_{3,3}\left(w;q^{-1}\right) = \psi_1\left(w^{-1};q\right).
\end{equation}
\end{theorem}

\begin{proof}[Proof of Theorem \ref{thm:g3_q<1}]
To prove the first equality stated in Theorem \ref{thm:g3_q<1}, we use the identities
\begin{align}\label{riden1}
R(w;q) & = \frac{1-w}{(q)_\infty} \sum_{n\in \mathbb Z} \frac{(-1)^n q^\frac{3n^2+n}{2}}{1-wq^n}, \\ \label{g3iden1}
g_3(w;q) &= \frac{1}{(q)_\infty} \sum_{n\in \mathbb Z} \frac{(-1)^n q^\frac{3n^2 + 3n}{2}}{1-wq^n},
\end{align}
the first of which is contained in the proof of Theorem 7.1 of \cite{zagierBourbaki}, and the second of which may be found in (3.3) of \cite{gmSurvey}. Using these, and the fact that $\sum_{n\in\Z} (-1)^n q^{(3n^2+n)/2} = (q)_\infty$, we obtain that
\begin{align*}
g_{3,1}(w;q) - g_{3,2}(w;q) &= \frac{1}{w} - \frac{1}{w(q)_\infty} \sum_{n\in\mathbb Z} (-1)^n q^\frac{3n^2 + n}{2} = 0.
\end{align*}

To prove the second equality stated in Theorem \ref{thm:g3_q<1}, we first rewrite
$$(1-w)g_{3,3}(w;q) = F\left(0,w;\frac{q}{w}\right).$$ We then use (\ref{123fine}) to deduce that $g_{3,3}(w;q) = g_{3,1}(w;q)$.
\end{proof}

\begin{proof}[Proof of Theorem \ref{thm:g3_q>1}]
Set $\rho=q^{-1}$. We use the fact that for $n \in \N_0$ we have
 \begin{equation}\label{randq}
(a;\rho)_n = (a^{-1};q)_n (-a)^n \rho^{\frac{n(n-1)}{2}}.
\end{equation}
We obtain, employing \eqref{ram11} with $a=-w^{-1}$,
\begin{eqnarray*}
g_{3,2}(w;\rho)&=&
-w^{-1}+\frac{1}{w(1-w)} \sum_{n\ge 0} \frac{\rho^{n^2}}{(w\rho, w^{-1}\rho;\rho)_n} =
-w^{-1} + \frac{1}{w(1-w)}\sum_{n\ge 0}\frac{q^{n} }{(wq, w^{-1}q)_n} \\
&=&  -w^{-1} -w^{-2} \sum_{n\geq 0} (-1)^n w^{-3n} q^\frac{n(3n+1)}{2} \left(1-w^{-2} q^{2n+1}\right) + S_2(w^{-1};q).
\end{eqnarray*}
We note that the second equality also yields convergence in the claimed domain. Questions of convergence follow similar in all cases, therefore we will not mention them anymore.

Identity (\ref{eqn:g3_1/qA}) is now deduced from the easily verified identity
\begin{align} \label{partialident}
\sum_{n\ge 0} (-1)^n w^{-3n}q^{\frac{n(3n+1)}{2}}\left(1-w^{-2}q^{2n+1}\right)
= \sum_{n\ge 0} \(\frac{12}{n}\) w^{\frac{-n+1}{2}}
q^{\frac{n^2-1}{24}}. \end{align}

To prove \eqref{eqn:g3_1/qB}, we use again (\ref{randq}) to obtain
\begin{align}\label{g33calc1}
g_{3,3}(w;\rho)=
\sum_{n\ge 0} \frac{w^{-n}\rho^n}{(w;\rho)_{n+1}} =
\sum_{n\ge0} \frac{w^{-2n-1}(-1)^{n+1}q^{\frac{n^2-n}{2}}}{(w^{-1})_{n+1}}.
\end{align}
Using \eqref{fineshift}, \eqref{by31}, and \eqref{partialident}, we find that
\begin{align*}
\sum_{n\ge 0} \frac{(-1)^n w^{-2n} q^{\frac{n^2-n}{2}}}{(w^{-1}q)_{n}} &=
\lim_{x\to \infty} F\(x, w^{-1}; \frac{w^{-2}}{q x}\)
=  \lim_{x\to \infty}\( \frac{1-w^{-1}}{1-\frac{w^{-2}}{q x}} +
\frac{w^{-1}-w^{-2}}{1-\frac{w^{-2}}{q x }} F\(x, w^{-1}; \frac{w^{-2}}{x}\) \) \\
&= \left(1-w^{-1}\right) + w^{-1}\left(1-w^{-1}\right) \sum_{n\ge 0} \frac{(-1)^n w^{-2n} q^{\frac{n(n+1)}{2}}}{(w^{-1} q)_n}  \\
&= \left(1-w^{-1}\right)\left(1+w^{-1}\sum_{n\ge 0} \(\frac{12}{n}\) w^{\frac{-n+1}{2}} q^{\frac{n^2-1}{24}}\right)
= (1-w) \psi_1(w^{-1};q).
\end{align*} Inserting this into (\ref{g33calc1}) now yields the result.
\end{proof}
\begin{remark} As is suggested in \cite{lz}, the partial theta function $\psi_1(\zeta;q)$, for $\zeta$ a root of unity,
is closely related to the \emph{shadow} of the mock theta function $g_3(\zeta;q)$ (see \cite{zagierBourbaki}, e.g., for a definition of this term and more on mock modular forms and \cite{kang} for the shadow of $g_3$).
%That is, by work of Kang in \cite{kang}, we have that
%$$q^{-1/24} \zeta^{-\frac{1}{2}} (1 + \zeta g_3(\zeta; q))$$
% is a mock modular form of weight $1/2$ with shadow
%proportional to the weight $3/2$ unary theta function \begin{align*}
% \sum_{n \in \Z} n \(\frac{12}{n}\)  \zeta^{\frac{n-1}{2}} q^{\frac{n^2-1}{24}},
% \end{align*} for roots of unity $\zeta \ne 1$.
%Note the similarity between this series and $\psi_1(\zeta;q)$.
\end{remark}
%%%%%%%%%%%%%%%%%%%%%%%%%%%%%
%%%%%%%%%%%SECTION 4%%%%%%%%%%%%
%%%%%%%%%%%%%%%%%%%%%%%%%%%%%
\section{The universal mock theta function $g_2$}\label{sec:g2sec}

 Just as in the case of the universal mock theta function $g_3(w;q)$ described in \S \ref{sec:g3sec}, we can give alternate $q$-hypergeometric series
expressions for the universal mock theta function $g_2(w;q)$ inside $\mathbb D$, and moreover, show that these different expressions are related by partial theta functions outside $\mathbb D$.  One particularly nice expression involves a combinatorial rank generating function.
To describe this, we let for $|q|< 1$, and $w\in \mathbb C\setminus \{0\}$ with $w \neq q^\ell \left(\ell \in \mathbb Z \setminus \{0\}\right)$,
$$O_2(w;q) := \sum_{n\ge 0} \frac{(-1)_n q^{\frac{n(n+1)}{2}}}{(wq, w^{-1}q)_n},
$$
denote the overpartition rank generating function (defined in Proposition 1.1 of \cite{lovejoy}).
We point out that the overpartition function
$O_2(w;q)$ is equal to McIntosh's function $K_2(w;q)$, and the universal mock theta function $g_2(w;q)$ is equal to McIntosh's function $H(w;q)$ (see \cite{mcintosh2}).

Analogous to the series $g_{3,j}(w;q), 1\leq j \leq 3$, we define for $|q|< 1$ and $w\in\C\setminus\{0\}$ with $w\neq\nolinebreak q^\ell(\ell\in\nolinebreak\mathbb Z)$
\begin{align*}
g_{2,1}(w;q) &:= g_2(w;q), \\
g_{2,2}(w;q)  &:=  \frac{1+w}{2w(1-w)} O_2(w;q) - \frac{1}{2w}.
\end{align*}
We also define for $|q|<1$ and $w\neq q^{-\ell}(\ell\in\N)$ and $|w|>1$ the function
\[
g_{2,3}(w;q) := - \frac{1+w}{2w^2}\sum_{n\ge 0} \frac{(-wq)_n}{(wq)_n} w^{-n} -\frac{1}{2w}.
\]
Our next result is an analogue to Theorem \ref{thm:g3_q<1} for $g_2(w;q)$ when $\abs{q}<1$.
\begin{theorem}\label{thm:g2_q<1}
For $\abs{q}<1$ and $w\in \C$, assuming the same restrictions as above, we have that
\begin{align}
g_{2,1}(w;q) = g_{2,2}(w;q) = g_{2,3}(w;q).
\end{align}
\end{theorem}
Next we will analyze these functions outside the unit disc.  We introduce partial theta functions for $|q|<1$, defined for $w\in\C$ as follows:
\begin{align}\label{eqn:g2partial}
\psi_3(w;q)&:=\sum_{n\geq 0}(-1)^{n+1} w^{2n+1}q^{n^2}, \\
\label{eqn:g2partial2}\psi_4(w;q) &:= \sum_{n \geq 1}w^{3n-2}q^{\frac{3n^2 -n}{2}}\left(1-wq^{n}\right).
\end{align}
Moreover if $w\neq 0$ and $w\neq -q^\ell (\ell\in\Z)$, we define
\begin{align}
S_1(w;q) &:= \frac{(-q)_\infty}{(-w^{-1}, -wq)_\infty} \left( w^{-1} \psi_1(w;q) + 1 \right), \end{align} and for $w\neq 0$ with $w\neq q^\ell (\ell\in\mathbb Z)$, we let
\begin{align} S_4(w;q) &:=\frac{(-q)_\infty}{(wq, w^{-1})_\infty} \psi_4(w;q).
\end{align}
In analogy to Theorem \ref{thm:g3_q>1}, we have the following relations between the series in Theorem \ref{thm:g2_q<1} and
partial theta functions.
\begin{theorem}\label{thm:g2_q>1}
For $\abs{q}<1$ and $w\in \C\setminus\{0\}$ with $w\neq q^\ell (\ell\in\Z)$, we have
\begin{align}
g_{2,1}\left(w;q^{-1}\right)
   &= -w^{-2}\psi_3\left(w;q\right) - w^{-1}
   +S_4(w;q),
   \label{eqn:g2_q>1A}
   \\ \label{eqn:g2_q>1B}
g_{2,2}\left(w;q^{-1}\right) &= \psi_3\left(w^{-1};q\right) + S_1(-w^{-1};q).
\end{align}
Moreover for $|q|<1$ and $w\in\C$ with $w\neq q^\ell(\ell\in\N)$ and $|w|>1$, we have
\begin{equation}\label{eqn:g2_q>1C}
g_{2,3}\left(w;q^{-1}\right) =   \psi_3\left(w^{-1};q\right).
\end{equation}
\end{theorem}
\begin{remark} It is not difficult to see that
$$\psi_{4}(w;q) = w^{-1} - w^{-2} + w^{-3} \psi_1(-w;q).$$
\end{remark}
\begin{remark} Note that $$\psi_3(w^{-1};q) + w^{-2} \psi_3(w;q) + w^{-1} = \sum_{n\in \mathbb Z} (-1)^{n+1} w^{2n-1} q^{n^2}$$ is the classical Jacobi theta function.  In this way, by Theorem \ref{thm:g2_q<1}, we have two series, $g_{2,1}(w;q)$ and $g_{2,2}(w;q)$, that are equal to each other, and define a mock theta function inside the unit disc.  On the other hand, by Theorem \ref{thm:g2_q>1}, outside of the disc these series are essentially each equal to half of a Jacobi theta function, but not the same half!
\end{remark}
\begin{proof}[Proof of Theorem \ref{thm:g2_q<1}]
By (3.2) of \cite{mcintosh2}, we have that $$g_{2,1}(w;q) =
\frac{(1+w)}{2w(1-w)} O_2(w;q) - \frac{1}{2w} = g_{2,2}(w;q).$$ On the other hand,  (\ref{122fine}) yields
$$ \frac{w}{w-1} \sum_{n\ge 0} \frac{(-1)_n q^{\frac{n(n+1)}{2}}}{(wq, w^{-1}q)_n}= F\(-w,w;\frac{1}{w}\) =
\sum_{n\ge 0} \frac{(-wq)_n}{(wq)_n} w^{-n},$$ which can be used to show that $g_{2,2}(w;q) = g_{2,3}(w;q)$.
\end{proof}

\begin{proof}[Proof of Theorem \ref{thm:g2_q>1}]
To establish \eqref{eqn:g2_q>1A},  we obtain, again using \eqref{randq},
\begin{align}\label{prelimg2}
g_2(w;\rho) = \sum_{n\geq 0} \frac{(-\rho;\rho)_n \ \rho^{\frac{n(n+1)}{2}}}{(w, w^{-1}\rho;\rho)_{n+1}} = \frac{q}{(1-wq)(1-w^{-1})} \sum_{n\geq 0} \frac{(-q)_n q^n}{(w^{-1}q, wq^2)_n}.
\end{align}
Next we apply \eqref{eqn:andrewsPartial} with $B=-q, a=-1/w,$ and $b = -wq$, and let $A\to 0$.  This gives that the rightmost expression in (\ref{prelimg2}) equals
\begin{align}\label{false}
\frac{q}{(1-w^{-1})}\sum_{n\geq 0} \frac{(w)_{n+1} (wq)^n}{(-wq)_{n+1}} + \frac{wq(-q)_\infty}{(wq, w^{-1} )_\infty} \sum_{n\geq 0} \frac{(-1)^n w^{2n} q^{\frac{n^2 + 3n}{2}}}{(-wq)_{n+1}}.
\end{align}
To the first summand in (\ref{false}) we apply (\ref{eqn:RogersFine}) with $a=wq, b=-wq^2$ and $t=wq$ to obtain
\begin{align}\label{sum1g2}-qw (1+qw)^{-1} F(w,-wq;wq)
=  \sum_{n\geq 1} (-1)^n w^{2n-1} q^{n^2}.\end{align}

Next we use (\ref{122fine}) to realize the second summand in (\ref{false}) as $$\frac{wq(-q)_\infty}{(wq,w^{-1})_\infty}F(-w, 0; -wq).$$  To this we apply (\ref{eqn:RogersFine}) with $a=-wq$ and $t=-wq$, and let $b\to 0$, yielding
\begin{align}\label{sum2g2} \frac{ (-q)_\infty}{(wq, w^{-1})_\infty} \sum_{n\geq 1} w^{3n-2} q^{\frac{3n^2 -n}{2}}(1-wq^{n}).\end{align}
Combining (\ref{sum1g2}) and (\ref{sum2g2}) gives \eqref{eqn:g2_q>1A}.

Next, to establish (\ref{eqn:g2_q>1B}), we use (\ref{randq}) and find
\begin{align}\label{44start}g_{2,2}(w;\rho) = \frac{1+w}{2w(1-w)} \sum_{n\geq 0} \frac{(-1;\rho)_n \rho^{\frac{n(n+1)}{2}}}{(w\rho, w^{-1}\rho; \rho)_n} - \frac{1}{2w}=\frac{1+w}{2w(1-w)} \sum_{n\geq 0} \frac{(-1)_n q^n}{(wq, w^{-1}q)_n} - \frac{1}{2w}.\end{align}
Then we apply an identity of Ramanujan, similar to (\ref{311Rand}), namely (3.18)$_R$ of \cite{andrewsPartial} with $c=-q^{-1}$, $a = - w$, and $b= - w^{-1}$ to see that
\begin{align}\label{secondg2id}\sum_{n\ge 0} \frac{(-1)_nq^{n}}{(wq, w^{-1}q)_n}
&= \frac{w-1}{w+1}\sum_{n\ge 0} \frac{(-1)^n q^{\frac{n(n+1)}{2}}
(-1)_nw^{-2n}}{\left(w^{-2}q^2;q^2\right)_n} +  \frac{(-1)_\infty}{w(wq, w^{-1}q)_\infty} \sum_{n\ge 0} \frac{(-1)^n q^{\frac{n(n+1)}{2}}
w^{-2n}}{(-w^{-1})_{n+1}}.
\end{align}
Using (\ref{by31}) and (\ref{partialident}), we have that the second $q$-hypergeometric series on the right hand side
equals $$\left(1+w^{-1}\right)^{-1}\sum_{n\ge 0} \(\frac{12}{n}\) (-w)^{\frac{1-n}{2}} q^{\frac{n^2-1}{24}}.$$ Next we use (\ref{122fine}) followed by (\ref{faaa}) to see that
\begin{align}\label{44tool}\sum_{n\ge 0}  \frac{(-1)^n q^{\frac{n(n+1)}{2}}
(-1)_nw^{-2n}}{(w^{-2}q^2;q^2)_n} = \left(1-w^{-1}\right)F\left(w^{-1},-w^{-1};w^{-1}\right) = 1+2\sum_{n\ge 1} (-1)^n w^{-2n} q^{n^2}.\end{align}  Using (\ref{secondg2id}) and (\ref{44tool}) in (\ref{44start}) together with (\ref{eqn:g3partial}) now
gives \eqref{eqn:g2_q>1B}.

Equation \eqref{eqn:g2_q>1C} follows by first using (\ref{randq}) to write
\begin{align}\label{g23Fid}g_{2,3}\left(w;\rho\right) = -\frac{(1+w)}{2w^2} \sum_{n\geq 0} \frac{(-w\rho;\rho)_n w^{-n}}{(w\rho;\rho)_n} - \frac{1}{2w} = -\frac{(1+w)}{2w^2}\sum_{n\geq 0} \frac{(-1)^n w^{-n} \left(-w^{-1}q\right)_n}{(w^{-1}q)_n} - \frac{1}{2w}.\end{align}
After applying (\ref{faaa}) to the rightmost series expression in (\ref{g23Fid}) and simplifying we obtain (\ref{eqn:g2_q>1C}).
\end{proof}
%\begin{remark}
%As with the function $g_3(w;q)$ if
% $\zeta  \ne 1$ is a root of unity, then by \cite{kang} $$1+ 2 \zeta g_2(\zeta; q)$$
%is a mock modular form of weight $1/2$
%with shadow equal to
% \begin{align}\label{g2shad} \sum_{n\in \Z} n \ (-1)^n  \zeta^{-2n} q^{n^2},\end{align} which is related to the partial theta functions $\psi_3(\zeta^{-1};q) $ as defined in (\ref{eqn:g2partial}).
%\end{remark}
%%%%%%%%%%%%%%%%%%%%%%%%%%%%%
%%%%%%%%%%%SECTION 5%%%%%%%%%%%%
%%%%%%%%%%%%%%%%%%%%%%%%%%%%%
\section{The universal mock theta function $K$}\label{sec:Ksec}
In this section, we study the universal mock theta function $K(w;q)$ as defined in (\ref{eqn:universalK}).  Analogous to the series $g_{k,j}(w;q), 2\leq k \leq 3, 1\leq j \leq 3$ defined in \S \ref{sec:g3sec} and \S \ref{sec:g2sec}, we consider for $|q|< 1$ and $w\neq 0$ and $w\neq q^{2\ell+1} (\ell\in\Z)$
$$
 K_1(w;q) := \sum_{n\ge 1} \frac{(-1)^{n-1} q^{n^2}(q;q^2)_{n-1}}{(w q, w^{-1}q; q^2)_{n}}.
$$
Moreover for $|q|<1$ and $w\neq q^{-2\ell+1}(\ell\in\N)$ and $|w|>|q|$, we define
$$
\kappa(w;q) := \sum_{n\geq 0} \frac{ q^{n+1}w^{-n} (wq^2;q^2)_n}{(wq;q^2)_{n+1}}.
$$
The functions $K(w;q)$ and $K_1(w;q)$ are discussed at length in relation to the function $K_2(w;q)=O_2(w;q)$ from \S \ref{sec:g2sec} in \cite{mcintosh2}.  We point out that $K(w;q)$, as given in (\ref{eqn:universalK}), is defined for $|q|<1$, and $w\in \mathbb C\setminus \{0\}$, where $w\neq q^{2\ell} \left(\ell \in \mathbb Z \setminus \{0\}\right).$

We define the following partial theta function:
$$
\psi_5(w;q) := \sum_{n\geq 0} \(\frac{n}{3}\)
(-w)^{n-1} q^{\frac{n^2-1}{3}}.
$$
If $w\neq 0$ and $w\neq q^{2\ell+1}(\ell\in\Z)$ let
$$
s_1(w;q) := \frac{\left(q;q^{2}\right)_\infty}{w\left(wq, w^{-1}q; q^2\right)_\infty} \left( w^{-1} \psi_1\left(w;q^2\right) + 1 \right).
$$
Moreover if $w\neq q^{2\ell} (\ell \in \Z \setminus\{0\})$, we define
$$
S_5(w;q) :=  \frac{w (q;q^2)_\infty}{(wq^2, w^{-1}q^2; q^2)_\infty} \psi_5(w;q).
$$
We offer the following results for these functions inside and outside the disc.

\begin{theorem}\label{thm:K1_q<1}
For $\abs{q}<1$, and under the same restrictions as above, we have
\begin{align}\label{kident}
\kappa(w;q) &=  K_1(w;q)   = -\frac{w}{(1-w)^2}\left(K(w;q)  - \frac{(q; q^2)_\infty^3 (q^2; q^2)_\infty}{(wq, w^{-1}q)_\infty}\right).
\end{align}
\end{theorem}
\begin{theorem}\label{thm:K1_q>1}
For $\abs{q}<1$ we have for $w\neq q^{2\ell+1}(\ell\in\N_0)$ with $|w|>1$
\begin{equation}\label{eqn:K1_q>1B}
\kappa\left(w;q^{-1}\right)
 = \frac{1}{1-w}\psi_2\left(w^{-1};q\right).
 \end{equation}
 If $w\neq 0$ and $w\neq q^{2\ell+1}(\ell\in\Z)$, we have
 \begin{equation}\label{eqn:K1_q>1A}
K_1\left(w;q^{-1}\right) =   \frac{w}{w-1} \left(\psi_2\left(w;q\right) +s_1(w;q)\right).
\end{equation}
Moreover if $w\neq 0$ and $w\neq q^{2\ell}(\ell\in\Z \setminus \{0\})$, we obtain
\begin{equation}\label{eqn:K_q>1}
 K\left(w;q^{-1}\right) = (1-w)\psi_2(w;q) + S_5(w;q).
\end{equation}
\end{theorem}
\begin{proof}[Proof of Theorem \ref{thm:K1_q<1}]
The fact that the rightmost expression stated in (\ref{kident}) equals $K_1(w;q)$ may be found in (5.3) of \cite{kang}.  To prove that
$K_1(w;q) = \kappa(w;q)$, we modify the proof given in \cite{andrewsMordell} of (3.6)$_R$ to accommodate the parameter $w$.  We have, applying (10.1) of \cite{sears} with $q$ replaced by $q^2$, $a=q^3/x, b=q, c=q^2, e=wq^3, f=w^{-1} q^3$, and letting $x\to 0$.
\begin{align*}
K_1(w;q) &= \frac{q}{(1-wq)(1-w^{-1}q)} \ \lim_{x\to 0} \ \sum_{n\geq 0} \frac{\left(q^3/x,q,q^2;q^2\right)_n x^n}{(wq^3,w^{-1}q^3,q^2;q^2)_n} = \sum_{n\geq 0} \frac{q^{n+1} w^{-n}(wq^2;q^2)_n }{(wq;q^2)_{n+1}}.
\end{align*}
\end{proof}

\begin{proof}[Proof of Theorem \ref{thm:K1_q>1}]
To prove \eqref{eqn:K1_q>1B}, we find, using (\ref{randq}), that
\begin{align}\label{kapinv1}\kappa(w;\rho) =   \sum_{n\ge 0} \frac{w^{-n} \rho^{n+1} \left(w\rho^2;\rho^2\right)_n}{\left(w\rho;\rho^2\right)_{n+1}} = - \frac{1}{w}
 \sum_{n\ge 0} \frac{w^{-n}  (w^{-1}q^{2};q^{2})_n}{(w^{-1}q;q^{2})_{n+1}}.\end{align}
We next apply the Rogers-Fine identity \eqref{eqn:RogersFine} with $q\to q^2$, and $a=w^{-1} q^2, b= w^{-1} q^3,$ $t=w^{-1}$.  From this and (\ref{kapinv1}) we find that
\begin{align*}\kappa(w;\rho) &=  \frac{1}{1-w} \sum_{n\geq 0} q^{2n^2 + n} w^{-2n} \left(1+ q^{2n+1} w^{-1}\right) \\
&= \frac{1}{1-w}  \left(\sum_{n\in 2\mathbb N_0} q^{\frac{n(n+1)}{2}} w^{-n} +\sum_{n\in 2\mathbb N_0 + 1} q^{\frac{n(n+1)}{2}} w^{-n} \right)\\
&= \frac{1}{1-w} \psi_2\left(w^{-1};q\right).\end{align*}

Next, to prove \eqref{eqn:K1_q>1A},
we note that by (\ref{randq}),
\begin{align}\label{k1invert}K_1(w;\rho) = \sum_{n\ge 0} \frac{(-1)^n \rho^{(n+1)^2} (\rho;\rho^2)_{n}}{(w\rho, w^{-1}\rho, \rho^2)_{n+1}} =
\sum_{n\ge 0} \frac{q^{2n+1} (q;q^{2})_{n}}{(wq, w^{-1}q; q^2)_{n+1}}.\end{align}
We apply
(\ref{311Rand}) to see that the rightmost expression in (\ref{k1invert}) equals
$$\frac{w}{w-1}\left(\sum_{n\geq 0} w^n q^{\frac{n(n+1)}{2}} - \frac{(q)_\infty \sum_{n\geq 0} (-1)^n w^{3n} q^{3n^2 + n} \left(1-w^2 q^{4n+2}\right)}{\sum_{n\in \mathbb Z} (-1)^n w^n q^{n^2}}\right).$$
 Equation \eqref{eqn:K1_q>1A} now follows from $(q)_\infty/(q^{2};q^2)_\infty = (q;q^2)_\infty$, \eqref{partialident}, and the identity
 \begin{align*}
\sum_{n\in \Z} (-1)^nw^n q^{n^2} &= \left(q^{2};q^2\right)_\infty\left(wq;q^{2}\right)_\infty \left(w^{-1}q;q^{2}\right)_\infty,
\end{align*}
which follows from the Jacobi triple product identity (see \cite{onobook} for example).

Finally, to prove (\ref{eqn:K_q>1}), we have again using (\ref{randq})
$$K(w;\rho) = \sum_{n\ge 0} \frac{(-1)^n \rho^{n^2} (\rho;\rho^2)_{n}}{(w\rho^2, w^{-1}\rho^2; \rho^2)_n} =
\sum_{n\ge 0} \frac{q^{2n} (q;q^{2})_n}{(wq^2, w^{-1}q^2; q^2)_n}.$$  To this we apply (3.14)$_R$ of \cite{andrewsPartial} to find that
$$K(w;\rho) = (1-w)\psi_2(w;q) + \frac{w (q;q^2)_\infty}{(wq^2, w^{-1}q^2; q^2)_\infty} \sum_{n\geq 0} (-1)^n w^{3n} q^{3n^2 + 2n} \left(1+wq^{2n+1}\right).$$
It is straightforward to see that
$$ \sum_{n=0}^\infty (-1)^n w^{3n} q^{3n^2+2n} \left(1+wq^{2n+1}\right) = \sum_{n\ge 0} \(\frac{n}{3}\)
(-w)^{n-1} q^{\frac{n^2-1}{3}},$$ and the result now follows.
\end{proof}

\begin{remark}
We note that $\psi_2 (w;q) + w^{-1} \psi_2 \left(w^{-1}; q\right) = \sum_{n\in\Z} w^n q^{\frac{n(n+1)}{2}}$ is a Jacobi theta function.
\end{remark}

\section*{Acknowledgements}
The authors thank Jeremy Lovejoy for several insightful comments.  The first author was partially supported by the Alfried Krupp-prize, and the second author is grateful for the support of National Science Foundation grant DMS-1049553.

%%%%%%%%%%%%%%%%
%%   BIBILIOGRAPHY          %%
%%%%%%%%%%%%%%%%

\end{document}